\documentclass[11pt]{amsart}

\usepackage{amsfonts,amsthm,amsmath,amssymb}
\usepackage{amscd}
\usepackage[comma,sort,compress, numbers]{natbib}

\usepackage{hyperref} 
\hypersetup{ colorlinks=true, linkcolor=blue, citecolor=blue,
  filecolor=blue, urlcolor=blue}

\newcommand{\weak}{\stackrel{w}{\longrightarrow}}
\newcommand{\prob}{\stackrel{P}{\longrightarrow}}
\newcommand{\eid}{\stackrel{d}{=}}
\newcommand{\one}{{\bf 1}}

\newcommand{\reals}{{\mathbb R}}
\newcommand{\bbr}{\reals}
\newcommand{\vep}{\varepsilon}
\newcommand{\bbz}{\protect{\mathbb Z}}
\newcommand{\bbn}{\protect{\mathbb N}}

\newtheorem{theorem}{Theorem}[section]
\newtheorem{corollary}{Corollary}[section]
\newtheorem{fact}{Fact}[section]
\newtheorem{lemma}{Lemma}[section]

\newtheorem{remark}{Remark}

\newtheoremstyle{example}{\topsep}{\topsep}%
     {}
     {}
     {\bfseries}
     {}
     {\newline}
     {\thmname{#1}\thmnumber{ #2}\thmnote{ #3}}
\theoremstyle{example}
\newtheorem*{example}{Example}

\def\ol{\overline}
\def\chs{\cite{chakrabarty:hazra:sarkar:2014}}

\def\Var{{\rm Var}}
\def\E{{\rm E}}
\def\esd{{\rm ESD}}

\DeclareMathOperator{\Tr}{Tr} 
\DeclareMathOperator{\rank}{Rank}

\numberwithin{equation}{section}

\begin{document}

\title[Hadamard product]{The Hadamard product and the free convolutions}
\dedicatory{Dedicated to Prof. B. V. Rao on his 70th birthday}
\date{November 12, 2015}

\author[A. Chakrabarty]{Arijit Chakrabarty}
\address{Theoretical Statistics and Mathematics Unit, Indian Statistical
Institute, New Delhi}
\email{arijit@isid.ac.in}



\begin{abstract}
It is shown that if a probability measure $\nu$ is supported on a closed subset of $(0,\infty)$, that is, its support is bounded away from zero, then the free multiplicative convolution of $\nu$ and the semicircle law is absolutely continuous with respect to the Lebesgue measure. For the proof, a result concerning the Hadamard product of a deterministic matrix and a scaled Wigner matrix is proved and subsequently used. As a byproduct, a result, showing that the limiting spectral distribution of the Hadamard product is same as that of a symmetric random matrix with entries from a mean zero stationary Gaussian process, is obtained.
\end{abstract}

\maketitle

\section{Introduction}\label{sec:intro} In a recent paper \cite{chakrabarty:hazra:2015}, it is shown that if $\nu$ is a probability measure such that $\nu([\alpha,\infty))=1$ for some $\alpha>0$ and $\int_0^\infty x\nu(dx)<\infty$, then the free multiplicative convolution of $\nu$ and the semicircle law, defined below in \eqref{eq.defsc}, is absolutely continuous. In that paper, it is conjectured that the result should be true without the assumption that the mean is finite, although the methodology of that paper does not allow the removal of this assumption. This is the main goal of the current paper. Theorem \ref{t2} shows that if the probability measure $\nu$ is supported on a subset of the positive half line, which is bounded away from zero, then the free multiplicative convolution of $\nu$ and the semicircle law has a non-trivial semicircle component in the sense of free additive convolution. In other words, there exists a probability measure $\eta$ such that 
\begin{equation}\label{eq.intro}
\nu\boxtimes\mu_1=\eta\boxplus\mu_\alpha\,,
\end{equation}
where $\mu_t$ is the semicircle law with standard deviation $t$, defined in \eqref{eq.defsc} and $\boxtimes$ and $\boxplus$ denote the free multiplicative and additive convolutions respectively. 
This is precisely the result proved in \cite{chakrabarty:hazra:2015}, albeit with the additional assumption that $\nu$ has finite mean. Theorem \ref{t2} and its corollary that $\nu\boxtimes\mu_1$ is absolutely continuous with respect to the Lebesgue measure, complement a corresponding result for the free additive convolution, proved in \cite{biane:1997}. 

The proof of Theorem \ref{t2} is via the analysis of random matrices of the type
\[
\left(f\left(\frac i{N+1},\frac j{N+1}\right)X_{i\wedge j,i\vee j}/\sqrt N\right)_{1\le i,j\le N}\,,
\]
where $f$ is a function on $(0,1)^2$ satisfying certain regularity properties, and $\{X_{i,j}:1\le i\le j\}$ is a family of i.i.d. standard normal random variables. This random matrix is studied in Section \ref{sec:hp}, and the observations are summarized in Theorem \ref{t1}. In Section \ref{sec:fc}, Theorem \ref{t1} is used to prove Theorem \ref{t2} which is the main result of this paper.

It turns out that the limiting spectral distribution obtained in Theorem \ref{t1} is same as that of a symmetric random matrix whose entries come from a stationary mean zero Gaussian process. Such random matrices were studied in \chs. The proof of this intriguing observation follows from equating the moments of the limiting spectral distributions obtained in the two models. This has been done in Section \ref{sec:sgp}, and the observation mentioned above is stated as Theorem \ref{t3}.

We conclude this section by pointing out the analogue of Theorem \ref{t2} in classical probability. The classical analogue is that if $X$ and $G$ are independent random variables, the latter following standard normal, then there exists a random variable $Y$ independent of $G$ such that
\[
XG\eid\sigma Y+G\,,
\]
if and only if 
\[
P(|X|\ge\sigma)=1\,.
\]
This can be proved using elementary probability tools. Theorem \ref{t2} is the free analogue of the \emph{if} part of the above result. The author believes that the free analogue of the \emph{only if} part is also true, that is, if \eqref{eq.intro} holds, then necessarily $\nu([\alpha,\infty))=1$. However, the methods of the current paper do not immediately prove the converse.

\section{The Hadamard product}\label{sec:hp}The following notations will be used throughout the paper. The $(i,j)$-th entry of a matrix $A$ will be denoted by $A(i,j)$. For  two $m\times n$ matrices $A$ and $B$, the Hadamard product of $A$ and $B$, denoted by $A\circ B$, is defined as
\[
(A\circ B)(i,j):=A(i,j)B(i,j),\,1\le i\le m,1\le j\le n\,.
\]
In other words, the Hadamard product is same as entrywise multiplication. 

Let $\mathcal R$ be the class of functions $f:(0,1)\times(0,1)\longrightarrow[0,\infty)$ such that 
\begin{enumerate}
\item\label{defr1} for all $0<\vep<1/2$, $f$ is bounded on $[\vep,1-\vep]^2$,
\item\label{defr2}  the set of discontinuities of $f$ in $(0,1)^2$ has Lebesgue measure zero,
\item\label{defr3}  and $f(x,y)=f(y,x)$ for all $(x,y)\in(0,1)^2$.
\end{enumerate}
Conditions \eqref{defr1} and \eqref{defr2} together are equivalent to assuming that $f$ is Riemann integrable on any compact subset of $(0,1)^2$, and hence the letter `R' has been used. However,  $\mathcal R$ is strictly larger than the class of Riemann integrable functions on $(0,1)^2$ satisfying \eqref{defr3}.

Fix $f\in\mathcal R$. For all $N\ge1$, define a $N\times N$ matrix $A_{f,N}$ by
\begin{equation}\label{eq.defan}
A_{f,N}(i,j):=f\left(\frac i{N+1},\frac j{N+1}\right),\,1\le i,j\le N\,.
\end{equation}
Let $\{X_{i,j}:1\le i\le j\}$ be a family of i.i.d. standard normal random variables, and let $W_N$ be a $N\times N$ scaled Wigner matrix formed by them. That is,
\begin{equation}\label{eq.defwn}
W_N(i,j):=N^{-1/2}X_{i\wedge j,i\vee j},\,1\le i,j\le N\,.
\end{equation}
Define
\begin{equation}\label{eq.defzn}
Z_N:=A_{f,N}\circ W_N,\,N\ge1\,.
\end{equation}

The main result of this section is Theorem \ref{t1} below, which studies the LSD of $Z_N$, as $N\to\infty$. Before stating that result, we need to recall a few combinatorial notions. The reader can find a detailed discussion on these topics in \cite{nica:speicher:2006}. For all $m\ge1$, let $NC_2(2m)$ denote the set of all \emph{non-crossing pair partitions} of $\{1,\ldots,2m\}$. Fix $m\ge1$, and $\sigma\in NC_2(2m)$. Let $(V_1,\ldots,V_{m+1})$ denote the \emph{Kreweras complement} of $\sigma$, that is the maximal partition $\ol\sigma$ of $\{\ol1,\ldots,\ol{2m}\}$ such that $\sigma\cup\ol\sigma$ is a non-crossing partition of $\{1,\ol1,\ldots,2m,\ol{2m}\}$. Note that the Kreweras complement of an element in $NC_2(2m)$ has exactly $(m+1)$ blocks, and hence is not a pair partition, although it is still non-crossing. For the sake of an unique labeling of the $V_i$'s, we require that if $1\le i<j\le m+1$, then the \emph{maximal} element of $V_i$ is smaller than that of $V_j$. Denote by ${\mathcal T}_\sigma$ the function from $\{1,\ldots,2m\}$ to $\{1,\ldots,m+1\}$ satisfying
\[
i\in V_{{\mathcal T}_\sigma(i)},\,1\le i\le2m\,.
\]
The above notations have been introduced in \chs. For $\sigma\in NC_2(2m)$ and any function $f:(0,1)^2\longrightarrow\bbr$, define a function $L_{\sigma,f}:(0,1)^{m+1}\longrightarrow\bbr$ by
\[
L_{\sigma,f}(x_1,\ldots,x_{m+1}):=\prod_{(u,v)\in\sigma}f^2\left(x_{{\mathcal T}_\sigma(u)},x_{{\mathcal T}_\sigma(v)}\right)\,.
\]

\begin{theorem}\label{t1}
There exists a non-random probability measure $\mu_f$ on $\bbr$, which is symmetric about zero, such that 
\begin{equation}\label{t1.eq1}
\esd\left(Z_N\right)\to\mu_f\,,
\end{equation}
as $N\to\infty$, weakly in probability. If, in addition, $f\in L^\infty\left((0,1)^2\right)$, then $\mu_f$ is compactly supported, and 
\begin{equation}\label{t1.eq2}
\int_\bbr x^{2n}\mu_f(dx)=\sum_{\sigma\in NC(2n)}\int_{(0,1)^{n+1}}L_{\sigma,f}(x_1,\ldots,x_{n+1})dx_1\ldots dx_{n+1}\,.
\end{equation}
\end{theorem}

\begin{example}
Fix $\alpha>0$. By an abuse of notation, let $\alpha(\cdot,\cdot)$ also denote the function on $(0,1)^2$, taking the constant value $\alpha$. Then, Theorem \ref{t1} reiterates the classical result of Wigner, that is,
\[
\esd(Z_N)\to\mu_\alpha\,,
\]
weakly in probability, as $N\to\infty$, where $\mu_\alpha$ is the semicircle law with standard deviation $\alpha$, that is, 
\begin{equation}\label{eq.defsc}
\mu_\alpha(dx)=\frac1{2\pi\alpha}\sqrt{4-\left(\frac x\alpha\right)^2}\,\one(|x|\le2\alpha)\,dx,\,x\in\bbr\,.
\end{equation}
Throughout the paper, whenever $\mu$ is used with a subscript which is a positive number as opposed to a function, it will refer to the probability measure defined above, and in particular, $\mu_1$ is the standard semicircle law.
\end{example}

The rest of this section is devoted to the proof of the above theorem.

\begin{lemma}\label{l1}
Suppose that $f\in L^\infty\left((0,1)^2\right)$, and let $m_n$ denote the right hand side of \eqref{t1.eq2} for $n\ge1$. Then,
\[
\limsup_{n\to\infty}m_n^{1/2n}<\infty\,.
\]
\end{lemma}

\begin{proof}Let
\[
M:=\|f\|_\infty\,,
\]
and notice that
\begin{eqnarray*}
m_n^{1/2n}&\le& M\left(\#NC_2(2n)\right)^{1/2n}\\
&\sim&2 M\,,
\end{eqnarray*}
as $n\to\infty$, the equivalence in the second line following from Stirling formula. This completes the proof.
\end{proof}

\begin{lemma}\label{l2}
If $f$ is bounded on $(0,1)^2$ in addition to being in $\mathcal R$, then the claims \eqref{t1.eq1} and \eqref{t1.eq2} of Theorem \ref{t1} hold. 
\end{lemma}

\begin{proof}
The proof is by the method of moments.  The assumptions \eqref{defr2} and that $f$ is bounded imply that $f$ is Riemann integrable on $(0,1)^2$. Usual combinatorial arguments will show that  for all $n\ge1$,
\begin{eqnarray}
\lim_{N\to\infty}\E\left(\frac1N\Tr(Z_N^n)\right)&=&\begin{cases}
m_{n/2},&n\text{ even}\,,\\
0,&n\text{ odd}\,,
\end{cases}\label{l2.eq1}\\
\text{and, }\lim_{N\to\infty}\Var\left(\frac1N\Tr(Z_N^n)\right)&=&0\,,\label{l2.eq2}
\end{eqnarray}
where $m_n$, as defined in Lemma \ref{l1}, is the right hand side of \eqref{t1.eq2}. Therefore, there exists a probability measure whose odd moments are zero and $2n$-th moment is $m_n$ for all $n\ge1$. Lemma \ref{l1} shows that any such probability measure is necessarily supported on a compact set, and hence by the Carleman's criterion, unique. Let $\mu_f$ denote the probability measure whose moments are as above. Then clearly, $\mu_f$ is symmetric about zero, and \eqref{t1.eq1} holds. The claim \eqref{t1.eq2} is automatic in this case. This completes the proof.
\end{proof}

\begin{proof}[Proof of Theorem \ref{t1}]
Fix $f\in\mathcal R$ and $k\ge2$. Set
\begin{equation}\label{eq.deffk}
f_k:=f\one_{[1/k,1-1/k]^2}\,.
\end{equation}
Define 
\begin{equation}\label{eq.defznk}
Z_{N,k}:=A_{f_k,N}\circ W_N,\,N\ge1\,.
\end{equation}
The condition \eqref{defr1} in the definition of $\mathcal R$ ensures that $f_k$ is bounded, and thus satisfies the hypotheses of Lemma \ref{l2}. Hence there exists a symmetric probability measure $\mu_{f_k}$ such that 
\[
\esd(Z_{N,k})\to\mu_{f_k}\,,
\]
as $N\to\infty$, and for all $n\ge1$,
\begin{equation}\label{t1.eq3}
\int_\bbr x^{2n}\mu_{f_k}(dx)=\sum_{\sigma\in NC_2(2n)}\int_{(0,1)^{n+1}}L_{\sigma,f_k}(x_1,\ldots,x_{n+1})dx_1\ldots dx_{n+1}\,.
\end{equation}
In view of Fact 4.3 in \chs, if it can be shown that for all $\vep>0$, 
\begin{equation}\label{t1.eq4}
\lim_{k\to\infty}\limsup_{N\to\infty}P\left(d(\esd(Z_{N,k}),\esd(Z_{N}))>\vep\right)=0\,,
\end{equation}
where $d$ denotes the L\'evy metric, then it will follow that there exists a probability measure $\mu_f$ such that \eqref{t1.eq1} holds, and
\begin{eqnarray}
\mu_{f_k}&\weak&\mu_f\text{ as }k\to\infty\,.\label{t1.eq6}
\end{eqnarray}
Proceeding towards the proof of \eqref{t1.eq4}, Theorem A.43, page 503, \cite{bai:silverstein:2010} shows that
\begin{eqnarray*}
d(\esd(Z_{N,k}),\esd(Z_{N}))&\le&\frac1N\rank\left(Z_{N,k}-Z_N\right)\\
&\le&\frac4k\frac{N+1}N\,.
\end{eqnarray*}
This shows \eqref{t1.eq4}, which in turn establishes \eqref{t1.eq1}. 

For showing \eqref{t1.eq2}, assume that $f\in{\mathcal R}\cap L^\infty\left((0,1)^2\right)$. Defining $f_k$ as above, it follows that $f_k\uparrow f$ as $k\to\infty$, and hence by the monotone convergence theorem, the right hand side of \eqref{t1.eq3} converges to that of \eqref{t1.eq2}. Denoting the latter by $m_n$, what follows is that
\[
\lim_{k\to\infty}\int_\bbr x^{2n}\mu_{f_k}(dx)=m_n\text{ for all }n\ge1\,.
\]
Lemma \ref{l1} shows that there exists a unique probability measure $\nu$ supported on a compact set such that 
\begin{equation}\label{t1.eq5}
\int_\bbr x^n\nu(dx)=\begin{cases}
m_{n/2},&n\text{ even}\,,\\
0,&n\text{ odd}\,.
\end{cases}
\end{equation}
Therefore, as $k\to\infty$, 
\[
\mu_{f_k}\weak\nu\,.
\]
Equating this with \eqref{t1.eq6} shows that $\nu=\mu_f$, and thus \eqref{t1.eq5} establishes \eqref{t1.eq2}. This completes the proof.
\end{proof}

We end this section with the following remark. 

\begin{remark}
The condition \eqref{defr3} in the definition of $\mathcal R$ is needed for the symmetry of the matrix $A_{f,N}$, and therefore for that of $Z_N$. The other two conditions -- \eqref{defr1} and \eqref{defr2}, are essential in that the claim in \eqref{t1.eq1} will fail if any one of them is dropped, or is replaced by $f\in L^\infty\left((0,1)^2\right)$. To see that, $f$ defined on $(0,1)^2$ by
\[
f(x,y):=\begin{cases}
\sqrt n,&\text{if }x=y=\frac mn\text{ where }m,n\in\bbn\text{ and }n\text{ is prime },\\
0,&\text{otherwise },
\end{cases}
\]
satisfies \eqref{defr2} and \eqref{defr3} and violates \eqref{defr1}. If $\{p_1,p_2,\ldots\}$ is the enumeration of the prime numbers in ascending order, then it is easy to see that $\esd(Z_{p_k})$ converges to standard normal as $k\to\infty$, while $\esd(Z_{2^N})$ converges to $\delta_{\{0\}}$, as $N\to\infty$. To see that \eqref{defr2} is essential, define another $f$ on $(0,1)^2$ by 
\[
f(x,y):=\begin{cases}
1, &\text{if both }x,y\text{ are dyadic rationals },\\
0, &\text{otherwise }.
\end{cases}
\]
For this $f$, $\esd(Z_{2^N})$ converges to the semicircle law, as $N\to\infty$, while $\esd(Z_{3^N})$ converges to $\delta_{\{0\}}$. Thus,  \eqref{defr1} - \eqref{defr3} are essential for \eqref{t1.eq1}.
\end{remark}

\section{The free additive and multiplicative convolutions}\label{sec:fc} The content of this section is the following theorem which is the main result of this paper. For stating the same, it is helpful to recall the notation in \eqref{eq.defsc}.

\begin{theorem}\label{t2}
If $\nu$ is a probability measure on $\bbr$ such that $\nu([\alpha,\infty))=1$ for some $\alpha>0$, then there exists a probability measure $\eta$ such that
\[
\nu\boxtimes\mu_1=\eta\boxplus\mu_\alpha\,.
\]
\end{theorem}

In view of Corollary 2 of \cite{biane:1997}, an immediate consequence of the above theorem is the following result.

\begin{corollary}\label{c1}
If $\nu$ satisfies the hypothesis of Theorem \ref{t2}, then $\nu\boxtimes\mu_1$ is absolutely continuous with respect to the Lebesgue measure.
\end{corollary}

We now proceed towards proving Theorem \ref{t2}. One of the main ingredients of the proof is the following fact which connects random matrices and the free additive and multiplicative convolutions. This is essentially the seminal discovery of Voiculescu \cite{voiculescu:1991}. As stated below, the fact is a corollary of Proposition 22.32, page 375, \cite{nica:speicher:2006}.

\begin{fact}\label{f1}
For $N\ge1$, let $W_N$ be the scaled Wigner matrix as defined in \eqref{eq.defwn}. Suppose that $Y_N$ is a $N\times N$ random matrix, such that as $N\to\infty$,
\begin{equation}\label{f1.assume}
 \frac1N\Tr(Y_N^k)\prob\int_\bbr x^k\mu(dx),\,k\ge1\,,
\end{equation}
for some compactly supported (deterministic) probability measure $\mu$. Furthermore, let the
families $(W_N:N\ge1)$ and $(Y_N:N\ge1)$ be independent.
Then, as $N\to\infty$,
\[
 \frac1N\E_{\mathcal F}\Tr\left[(W_N+Y_N)^k\right]\prob\int_\bbr
x^k\mu\boxplus\mu_1(dx)\text{ for all }k\ge1\,,
\]
where ${\mathcal F}:=\sigma(Y_N:N\ge1)$ and $\E_{\mathcal F}$ denotes the
conditional expectation with respect to $\mathcal F$. If in addition, the support of $\mu$ is contained in $[0,\infty)$, then, as $N\to\infty$,
\[
 \frac1N\E_{\mathcal F}\Tr\left[(W_NY_N)^k\right]\prob\int_\bbr
x^k\mu\boxtimes\mu_1(dx)\text{ for all }k\ge1\,.
\]
\end{fact}

The first step towards proving Theorem \ref{t2} is the following lemma.

\begin{lemma}\label{fc.l1}
If $f\in\mathcal R$ and $\alpha>0$, then
\[
\mu_{f+\alpha}=\mu_{\sqrt{f^2+2\alpha f}}\boxplus\mu_\alpha\,.
\]
\end{lemma}

\begin{proof}
Fix $k\ge2$, and let $f_k$ be as in \eqref{eq.deffk}. Define
\[
g(\cdot):=\sqrt{f_k^2(\cdot)+2\alpha f_k(\cdot)}\,.
\]
Our first task is to show that
\begin{equation}\label{fc.l1.eq4}
\mu_{f_k}=\mu_g\boxplus\mu_\alpha\,.
\end{equation}
To that end, denote for $N\ge1$, the matrices
\begin{eqnarray*}
U_N&:=&A_{f_k+\alpha,N}\circ W_N\,,\\
V_N&:=&A_{g,N}\circ W_N\,,
\end{eqnarray*}
where $A_{f,N}$ and $W_N$ are as defined in \eqref{eq.defan} and \eqref{eq.defwn} respectively. For every $N\ge1$, let $Y_N$ be a copy of $W_N$, independent of the $\sigma$-field 
\[
{\mathcal F}:=\sigma\left(X_{i,j}:1\le i\le j\right)\,.
\]
The key observation in the proof of the lemma is that
\begin{equation}\label{fc.l1.eq1}
U_N\eid V_N+\alpha Y_N,\,N\ge1\,.
\end{equation}
Since the functions $f_k$ and $g$ are both bounded, it follows by Theorem \ref{t1} that $\mu_{f_k}$ and $\mu_g$ are compactly supported, and hence so is $\mu_g\boxplus\mu_\alpha$. Thus, for showing \eqref{fc.l1.eq4}, it suffices to check that
\begin{equation}\label{fc.l1.eq5}
\int_\bbr x^n\mu_{f_k}(dx)=\int_\bbr x^n(\mu_g\boxplus\mu_\alpha)(dx)\text{ for all }n\ge1\,.
\end{equation}
It follows from \eqref{l2.eq1} and \eqref{l2.eq2} in the proof of Lemma \ref{l2} that for a fixed $n\ge1$,
\begin{equation}
\frac1N\Tr\left(V_N^n\right)\prob\int_\bbr x^n\mu_{g}(dx)\,,\label{fc.l1.eq6}
\end{equation}
and that
\begin{eqnarray}
\lim_{N\to\infty}\E\left(\frac1N\Tr(U_N^n)\right)&=&\int_\bbr x^n\mu_{f_k+\alpha}(dx)\,,\label{fc.l1.eq2}\\
\lim_{N\to\infty}\Var\left(\frac1N\Tr(U_N^n)\right)&=&0\,.\label{fc.l1.eq3}
\end{eqnarray}
Fact \ref{f1} applied to \eqref{fc.l1.eq6} implies that
\begin{equation}\label{fc.l1.eq7}
\frac1N\E_{\mathcal F}\Tr\left((V_N+\alpha Y_N)^n\right)\prob\int_\bbr x^n\mu_{g}\boxplus\mu_\alpha(dx)\,.
\end{equation}
On the other hand, \eqref{fc.l1.eq2} implies that
\[
\E\left[\frac1N\E_{\mathcal F}\Tr\left(U_N^n\right)\right]\to\int_\bbr x^n\mu_{f_k+\alpha}(dx)\,,
\]
and \eqref{fc.l1.eq3} implies that
\[
\Var\left[\frac1N\E_{\mathcal F}\Tr\left(U_N^n\right)\right]\le\Var\left(\frac1N\Tr(U_N^n)\right)\to0\,,
\]
thereby establishing that
\[
\frac1N\E_{\mathcal F}\Tr(U_N^n)\prob\int_\bbr x^n\mu_{f_k+\alpha}(dx)\,,
\]
as $N\to\infty$. Equating the above with \eqref{fc.l1.eq7} in view of \eqref{fc.l1.eq1} implies \eqref{fc.l1.eq5}, and thereby establishes \eqref{fc.l1.eq4}.

A restatement of \eqref{fc.l1.eq4} is that
\begin{equation}\label{fc.l1.eq8}
\mu_{f_k}=\mu_{\sqrt{f_k^2+2\alpha f_k}}\boxplus\mu_\alpha\,.
\end{equation}
Equation \eqref{t1.eq6} in the proof of Theorem \ref{t1} implies that the left hand side converges to $\mu_f$ as $k\to\infty$, and a similar argument shows that 
\[
\mu_{\sqrt{f_k^2+2\alpha f_k}}\weak\mu_{\sqrt{f^2+2\alpha f}}\,.
\]
An appeal to Proposition 4.13 of \cite{bercovici:voiculescu:1993} shows that the right hand side of \eqref{fc.l1.eq8} converges to $\mu_{\sqrt{f^2+2\alpha f}}\boxplus\mu_\alpha$, and thus completes the proof.
\end{proof}

\begin{lemma}\label{fc.l2}
Suppose that there exists a function $r:(0,1)\longrightarrow[0,\infty)$ such that 
\[
f(x,y)=r(x)r(y),\text{ for all }x,y\in(0,1)\,.
\]
Assume furthermore, that $r$ is Riemann integrable on any compact subset of $(0,1)$. Then, 
\[
\mu_f=\nu\boxtimes\mu_1\,,
\]
where $\nu$ is the law of $r^2(U)$, $U$ being a standard uniform random variable defined on some probability space.
\end{lemma}

\begin{proof}
In a way, the proof is analogous to that of Lemma \ref{fc.l1}. As was the case there, we first prove the result for $r$ bounded. So assume that. As a consequence, $r^n$ is Riemann integrable on $(0,1)$, for all $n\ge1$.  For $N\ge1$, let $R_N$ be the $N\times N$ diagonal matrix with diagonal entries $r\left(\frac1{N+1}\right),\ldots,r\left(\frac N{N+1}\right)$. Clearly,
\begin{equation}\label{fc.l2.eq1}
\lim_{N\to\infty}\frac1N\Tr\left(R_N^n\right)=\int_0^1r^n(x)dx=\int_\bbr x^{n/2}\nu(dx),\,n\ge1\,.
\end{equation}

Let $Z_N$ be as in \eqref{eq.defzn} with this particular $f$. Then, it is easy to see that
\[
Z_N=R_NW_NR_N,\,N\ge1\,.
\]
Thus, for a fixed $n\ge1$, 
\begin{eqnarray*}
\frac1N\E\Tr\left(Z_N^n\right)&=&\frac1N\E\Tr\left(\left(R_N^2W_N\right)^n\right)\\
&\to&\int_\bbr x^n(\nu\boxtimes\mu_1)(dx)\,,
\end{eqnarray*}
as $N\to\infty$, the second line following from an application of Fact \ref{f1} to \eqref{fc.l2.eq1}. Equation \eqref{l2.eq1} establishes that 
\[
\lim_{N\to\infty}\frac1N\E\Tr\left(Z_N^n\right)=\int_\bbr x^n\mu_f(dx),\,n\ge1\,.
\]
These two equations prove the result for the case when $r$ is bounded. When $r$ is possibly unbounded, similar arguments as before, for example those leading to the proof of Lemma \ref{fc.l1} from \eqref{fc.l1.eq8}, with an appeal to Corollary 6.7 of \cite{bercovici:voiculescu:1993} complete the proof.
\end{proof}

\begin{proof}[Proof of Theorem \ref{t2}]
Let 
\[
r(x):=\sqrt{\inf\{y\in\bbr: x\le\nu(-\infty,y]\}},\,0<x<1\,.
\]
Define
\[
f(x,y):=r(x)r(y),\,0<x,y<1\,,
\]
and $Z_N$ by \eqref{eq.defzn}. Clearly, $f\in\mathcal R$, and the assumption that $\mu([\alpha,\infty))=1$ implies that 
\[
f(x,y)\ge\alpha\text{ for all }x,y\in(0,1)\,.
\]
Lemma \ref{fc.l2} implies that
\begin{eqnarray*}
\nu\boxtimes\mu_1&=&\mu_f\\
&=&\mu_{(f-\alpha)+\alpha}\\
&=&\eta\boxplus\mu_\alpha\,,
\end{eqnarray*}
where $\eta:=\mu_{\sqrt{(f-\alpha)^2+2\alpha(f-\alpha)}}$, the last equality following by Lemma \ref{fc.l1}. This completes the proof.
\end{proof}

\begin{remark}
As was remarked in Section \ref{sec:intro}, the author believes that the converse of Theorem \ref{t2} is true, although how to prove it is unclear. However, a weaker claim can be made, namely if the probability measure $\nu\boxtimes\mu_1$ does not charge singletons, then $\nu(\{0\})=0$. This follows easily from the observation that
\[
\nu\boxtimes\mu_1(\{0\})=\nu(\{0\})\,,
\]
a fact which can easily be proven by the random matrix argument, or otherwise; see Proposition 2 in \cite{arizmendi:abreu:2009}. 
\end{remark}

\section{Stationary Gaussian processes}\label{sec:sgp} In this section, we show that the LSD of the Hadamard product matches with that of a random matrix model whose entries come from a stationary mean zero Gaussian process. Fix $g\in\mathcal L^1\left((0,1)^2\right)$, such that
\[
g(1-x,1-y)=g(x,y)\ge0\text{ for all }x,y\,.
\]
Define
\[
R(m,n):=\int_0^1\int_0^1 e^{2\pi\iota(mx+ny)}g(x,y)dxdy,\,m,n\in\bbz\,,
\]
where $\iota:=\sqrt{-1}$. Clearly, the right hand side is real for all $m,n$, and there exists a mean zero stationary Gaussian process $(G_{i,j}:i,j\in\bbz)$ whose covariance function is $(R(m,n):m,n\in\bbz)$, that is,
\[
\E\left(G_{i,j}G_{i+m,j+n}\right)=R(m,n)\text{ for all }i,j,m,n\in\bbz\,.
\]
For $N\ge1$, define a $N\times N$ random matrix $T_N$ by
\[
T_N(i,j):=\frac{G_{i,j}+G_{j,i}}{\sqrt{N}},\,1\le i,j\le N\,.
\]
Define
\[
f(x,y):=\sqrt{g(x,1-y)+g(1-y,x)},\,(x,y)\in(0,1)^2\,.
\]
We assume that $f\in\mathcal R$ which is as in Section \ref{sec:hp}. Note that \eqref{defr3} in the definition of $\mathcal R$ holds automatically for $f$.
The following theorem is the main result of this section.

\begin{theorem}\label{t3}
As $N\to\infty$,
\[
\esd(T_N)\to\mu_f\,,
\]
weakly in probability, where $\mu_f$ is as in Theorem \ref{t1}.
\end{theorem}

\begin{proof}
We first show that the result is true when in addition $\|g\|_\infty<\infty$, a consequence of the assumption being that
\[
f\in{\mathcal R}\cap L^\infty\left((0,1)^2\right)\,.
\]
By Theorems 2.1 and 2.4 of \chs, it follows that
\[
\esd(T_N)\to\nu\,,
\]
weakly in probability, where $\nu$ is a probability measure which is symmetric about zero, and for $m\ge1$, its $2m$-th moment is
\begin{eqnarray*}
&&\int_\bbr x^{2m}\nu(dx)\\
&=&\sum_{\sigma\in NC_2(2m)}(2\pi)^{m-1}\int_{(-\pi,\pi)^{m+1}}\prod_{(u,v)\in\sigma}\Biggl[\frac{g\left(\frac12+\frac{x_{{\mathcal T}_\sigma(u)}}{2\pi},\frac12-\frac{x_{{\mathcal T}_\sigma(v)}}{2\pi}\right)}{4\pi^2}\\
&&\,\,\,\,\,\,\,\,\,\,\,+\frac{g\left(\frac12-\frac{x_{{\mathcal T}_\sigma(v)}}{2\pi},\frac12+\frac{x_{{\mathcal T}_\sigma(u)}}{2\pi}\right)}{4\pi^2}\Biggl]dx_1\ldots dx_{m+1}\\
&=&\sum_{\sigma\in NC_2(2m)}\int_{(0,1)^{m+1}}L_{\sigma,f}(x_1,\ldots,x_{m+1})dx_1\ldots dx_{m+1}\\
&=&\int_\bbr x^{2m}\mu_f(dx)\,,
\end{eqnarray*}
the equality in the last line following from Theorem \ref{t1}. This shows that $\nu=\mu_f$ and thus completes the proof for the case when $\|g\|_\infty<\infty$.

Now suppose that $g$ is any function satisfying the hypothesis of the theorem. For all $0\le\vep<1/2$, and define
\begin{eqnarray*}
g_\vep&:=&g\one_{[\vep,1-\vep]^2}\,,\\
f_\vep(x,y)&:=&\sqrt{g_\vep(x,1-y)+g_\vep(1-y,x)},\,(x,y)\in(0,1)^2\,.
\end{eqnarray*}
For all $0\le\vep<1/2$, let $T_N^\vep$ denote the $N\times N$ random matrix obtained from $g_\vep$, in the same way as $T_N$ is obtained from $g$.
Theorem 2.1 and Lemma 4.3 of \chs imply that for all $0\le\vep<1/2$, there exists a probability measure $\nu_\vep$ such that
\begin{equation}\label{t3.eq1}
\esd(T_N^\vep)\to\nu_\vep\,,
\end{equation}
weakly in probability, as $N\to\infty$, and
\begin{equation}\label{t3.eq2}
\nu_\vep\weak\nu_0\,\text{, as }\vep\downarrow0\,.
\end{equation}
 Since $f\in\mathcal R$, it follows that for $0<\vep<1/2$, $f_\vep$ is an  $L^\infty$ function, for which  Theorem \ref{t3} has already been shown to be true. Thus,
\begin{equation}\label{t3.eq3}
\nu_\vep=\mu_{f_\vep},\,0<\vep<1/2\,.
\end{equation}
In the proof of Theorem \ref{t1}, it has been shown that
\[
\mu_{f_\vep}\weak\mu_f\,\text{, as }\vep\downarrow0\,.
\]
In conjunction with \eqref{t3.eq2} and \eqref{t3.eq3}, this implies that
\[
\mu_f=\nu_0\,.
\]
An appeal to \eqref{t3.eq1} with $\vep=0$ completes the proof.
\end{proof}

\section*{Acknowledgment}
The author is grateful to Rajat Subhra Hazra for helpful discussions.


\end{document}